\documentclass[10pt]{amsart}\oddsidemargin -1mm \evensidemargin -1mm \topmargin -15mm\headheight5mm \headsep 3.5mm \textheight 245mm\textwidth 170mm 
\usepackage{amsfonts, amssymb, amsthm, amsmath, euscript, hhline}
\usepackage{stmaryrd}
\theoremstyle{definition}\newtheorem{theorem}{Theorem}[section]\newtheorem{lemma}[theorem]{Lemma}\newtheorem{claim}[theorem]{Claim}\newtheorem{proposition}[theorem]{Proposition}\newtheorem{corollary}[theorem]{Corollary}
\DeclareMathOperator{\Ue}{\mathrm{U}}\DeclareMathOperator{\Ze}{\mathrm{Z}}

\newcommand{\apro}[1]{{\rm\setcounter{AP}{#1}\roman{AP}}}

\newcommand{\adj}{\mathrm{ad~}}
\newcounter{AP}
\begin{document}
\title{On Gelfand-Kirillov conjecture for some W-algebras}\maketitle
\begin{abstract}Consider the W-algebra $W$ attached to the smallest nilpotent orbit in a simple Lie algebra $\frak g$ over an algebraically closed field of characteristic 0. We show that if an analogue of the Gelfand-Kirillov conjecture holds for such a W-algebra then it holds for the universal enveloping algebra $\Ue(\frak g)$. This together with a result of A.~Premet implies that the analogue of the Gelfand-Kirillov conjecture fails for some $W$-algebras attached to some nilpotent orbits in Lie algebras of types $B_n~(n\ge 3)$, $D_n~(n\ge 4)$, $E_6, E_7, E_8$, $F_4$.


\end{abstract}
\section{Introduction}

Classical works, see for example~\cite[p. 486]{Cn}, show that a significant class of non-commutative algebras have a quotient field which is a non-commutative skew field. In this framework it is natural to ask whether or not such a skew field is isomorphic to a quotient field of a suitable Weyl algebra over some base commutative field. A more precise version of this question is known as the Gelfand-Kirillov conjecture: ``Whether or not the quotient field of the universal enveloping algebra of any algebraic Lie algebra is isomorphic to a quotient field of some Weyl skew field?''. In this paper we study a similar question for some W-algebras.

The solution of the original Gelfand-Kirillov conjecture for Lie algebras of type A and some other cases was settled by I.~Gelfand and A.~Kirillov themselves~\cite{GK1, GK2} and is positive. A version of this problem for W-algebras attached to type A Lie algebras was considered in~\cite{FMO} where the authors provides a positive solution to the corresponding problem.

In his paper~\cite{Pr2} A.~Premet shows that the Gelfand-Kirillov conjecture fails for $\Ue(\frak g)$ if $\frak g$ is simple and $\frak g$ is not of type $A_n$, $C_n$ or $G_2$. Another result of the same author~\cite{Pr1} shows that for a simple Lie algebra $\frak g$ we have that $\Ue(\frak g)$ is ``almost equal'' to the tensor product of some W-algebra with a suitable Weyl algebra.

The goal of this paper is to modify the result of~\cite{Pr1}, i.e. to show that the quotient field of $\Ue(\frak g)$ is isomorphic to the quotient field of the tensor product of the same W-algebra with a suitable Weyl algebra. This together with results of~\cite{Pr2} implies that the Gelfand-Kirillov conjecture fails for some W-algebras. It worth to mention that such W-algebras are deeply studied in~\cite{Pr1} and explicit generators and relations are known for them.

From now on the base field for all objects is an algebraically closed field $\mathbb F$ of characteristic 0.
\section{W-algebras}\label{SW}

A W-algebra $\Ue(\frak g, e)$ is some finitely generated algebra attached to a semisimple Lie algebra $\frak g$ and an $\frak{sl}_2$-triple $\{e, h, f\}$ inside $\frak g$ (see for example~\cite{Pr1}). The isomorphism class of such an algebra $\Ue(\frak g, e)$ depends only on the conjugacy class of $e$. We are particularly interested in W-algebras attached to an element $e$ from the smallest possible nonzero nilpotent orbit in $\frak g$. We take an explicit generators and relations presentation for such W-algebras~\cite[Theorem 1.1]{Pr1} and modify the notation of~\cite{Pr1} a little. Namely to a simple Lie algebra $\frak g$ we attach a semisimple Lie algebra $\frak g(0)$ with a $\frak g(0)$-module $\frak g(1)$. Then the algebra $H$ (this is a notation of~\cite{Pr1} for such W-algebras) would be generated by $\frak g(0), \frak g(1)$ and an additional element $C$ subject to the following relations:

(\apro{1}) $xy-yx=[x, y]$ for all $x, y\in\frak g(0)$, where $[x, y]$ is the Lie bracket of $\frak g(0)$;

(\apro{2}) $xy-yx=x\cdot y$ for all $x\in\frak g(0), y\in\frak g(1)$, where $x\cdot y$ is the action operator of element $x\in\frak g(0)$ applied to $y\in\frak g(1)$;

(\apro{3}) $C$ is central in $H$;

(\apro{4}) some formula $xy-yx=f(x, y)\in\Ue(\frak g(0))$ for $x, y\in\frak g(1)$, see~\cite[Theorem 1.1]{Pr1}.\\
Below we write explicitly Lie algebras $\frak g(0)$ and $\frak g(0)$-modules $\frak g(1)$ for all simple Lie algebras $\frak g$:
$$\mbox{Table~1}$$
$$\begin{array}{|c|c|c|c|c|c|c|c|c|c|}\hline \frak g&A_n&B_n&C_n&D_n&E_6&E_7&E_8&F_4&G_2\\\hline
\frak g(0)&A_{n-2}\oplus\mathbb F& B_{n-2}\oplus A_1&C_{n-1}&D_{n-2}\oplus A_1& A_5&D_6&E_7&C_3&A_1\\\hline
\frak g(1)&R(\pi_1)\oplus R(\pi_1)^*& R(\pi_1)\otimes R(\pi_1)& R(\pi_1)&R(\pi_1)\otimes R(\pi_1)& R(\pi_3)& R(\pi_5)& R(\pi_1)&R(\pi_3)&R(3\pi_1)\\\hline
\end{array}$$
(for Lie algebras and their representations we use notation of~\cite{OV}).
\subsection{Quotient fields and W-algebras} Most associative algebras without zero-divisors affords some embedding into a skew field, but this does not imply that the universal skew field of such an algebra exists (for details see for example~\cite[Chapter 7~\and~p.~486]{Cn}). There is a notion of the left (elements of the form $a^{-1}b$) and the right (elements of the form $ab^{-1}$) skew fields and they do coincide if both exist. This is the case if an algebra in question has no zero-divisors and has finite Gelfand-Kirillov dimension. All algebras in this paper (and thus all $W$-algebras) have no zero-divisors and have finite Gelfand-Kirillov dimension and thus they have the quotient field which is both the left and the right quotient field and is universal in an appropriate sense.

We now are ready to formulate a precise version of our main result.
\begin{theorem}\label{T1}The quotient fields of $\Ue(\frak g)$ and $H\otimes W_d$ are isomorphic, where $W_d:=\mathbb F[z_1,..., z_d, \partial_1,...,\partial_d]$ is a Weyl algebra and $d:=\dim\frak g(1)+1$.\end{theorem}
The proof of Theorem~\ref{T1} goes as follows. First we recall in Section~\ref{SinvW} a result of A.~Premet which states that $\Ue(\frak g)$ is ``almost birationally isomorphic'' to the tensor product of $H$ with a Weyl algebra. The precise statement of this fact goes through some involutions of a Weyl algebra and $H$. We provide another presentation of these involutions in Section~\ref{SinvH} and then complete the proof in Section~\ref{SprT}.

Further we will write $A\approx B$ if the quotient fields of two associative algebras $A$ and $B$ are isomorphic.

\section{Involutions of W-algebras}\label{SinvW} Algebra $H$ has an involution $\sigma$ (see~\cite[2.2]{Pr1}) the action of which on generators is given by the formulas:

(\apro{1}) $\sigma(x)=x$ if $x\in\frak g(0)$ or $x=C$;

(\apro{2}) $\sigma(x)=-x$ if $x\in\frak g(1)$.

The following lemma is a first step in a proof of Theorem~\ref{T1}.
\begin{lemma}We have that $$\Ue(\frak g)\approx (W_d\otimes H)^{\tau\otimes\sigma},$$ where $W_d$ is a Weyl algebra over $\mathbb F$ generated by $z_1,..., z_d, \partial_1,..., \partial_d$ satisfying relations
$$[z_i, \partial_j]=\delta_{ij}~\hspace{10pt}(1\le i, j\le d),$$ $\tau$ is an involution on $W_d$ such that
$$\tau(z_i)=-z_i, \tau(\partial_i)=-\partial_i\hspace{10pt}(1\le  i\le d).$$\end{lemma}
\begin{proof} This follows from the result of~\cite[1.5]{Pr1} and the observation that $\mathbb F[h]*\langle\Delta\rangle\approx\mathbb F[z_1, \partial_1]$ via maps

$$h\leftrightarrow z_1,\hspace{10pt}\Delta\leftrightarrow z_1\partial_1.$$\end{proof}

The next lemma is a useful refinement of the previous one.
\begin{lemma}\label{Ls1}We have that $$\Ue(\frak g)\approx W_{d-1}\otimes(\mathbb F[z_d, \partial_d]\otimes H)^{\tau'\otimes \sigma},$$ where $\mathbb F[z_d, \partial_d]$ is a Weyl algebra in one variable and involution $\tau'$ is determined by the formulas $$\tau'(z_d)=-z_d,\hspace{10pt} \tau'(\partial_d)=-\partial_d.$$\end{lemma}
\begin{proof}We consider a birational automorphisms of $W_d$ determined by a pair of maps:
$$z_i\leftrightarrow \frac{z_i}{z_d} (i< d),\hspace{10pt} z_d\leftrightarrow z_d$$
$$\partial_i\leftrightarrow z_d\partial_i (i<d),\hspace{10pt} \partial_d\leftrightarrow \partial_d+\frac1{z_d}(z_1\partial_1+...+z_{d-1}\partial_{d-1}).$$
One can check that the involution $\tau$ on $W_d$ is equivalent via this automorphism to the involution $\tau'$ on $W_{d-1}\otimes\mathbb F[z_d, \partial_d]$ such that $\tau'$ preserves the first factor pointwise and the action of $\tau'$ on $\mathbb F[z_d, \partial_d]$ is given by the formula
$$\tau'(z_d)=-z_d,\hspace{10pt} \tau'(\partial_d)=-\partial_d.$$

Thus we have that $$\Ue(\frak g)\approx (W_d\otimes H)^{\tau'\otimes\sigma}=W_{d-1}\otimes(\mathbb F[z_d, \partial_d]\otimes H)^{\tau'\otimes \sigma}.$$\end{proof}
We would like to mention that a philosophically similar decomposition of $\Ue(\frak g)$ into a tensor product of W-algebra with a Weyl algebra modulo some completion exists for any $W$-algebra, see~\cite{Lo}.

\section{Involutions of $H$ and $\frak{sl}_2$-triples.}\label{SinvH}In this section we relate $\sigma: H\to H$ with some $\frak{sl}_2$-triple in $\frak g(0)$. Namely we prove the following lemma.
\begin{lemma}\label{Lsl2inv} Let $H$ and $\sigma$ be as in Sections~\ref{SW},~\ref{SinvW} and assume that $\frak g$ is not of type $A$. Then there exists an $\frak{sl}_2$-triple $\{e, h, f\}$ such that

(1) the adjoint action of $\{e, h, f\}$ on $H$ is locally finite and thus it could be integrated to the action of an algebraic group $SL_2(\mathbb F)$,

(2) the action of $\sigma$ on $H$ coincides with the action of the only non-unit element of the center of $SL_2(\mathbb F)$ on $H$.\end{lemma}
\begin{proof} We claim that a suitable $\frak{sl}_2$-triple is an $\frak{sl}_2$-triple with a generic $e$. Thus we pick an $\frak{sl}(2)$-triple $\{e, h, f\}$ with a generic $e$. It is clear that such an $\frak{sl}_2$-triple preserve both spaces $\frak g(0)$ and $\frak g(1)$, and commute with $C$. Thus $\{e, h, f\}$ acts locally finitely on $H$. Statement (2) follows from a case-by-case checking via Table~1.\end{proof}

Now we will use an $\frak{sl}_2$-triple in $H$ to ``almost'' decompose $H$ as a tensor product of two algebras. We do this in several steps and the first one is as follows.

\subsection{Some localizations of $H$} The goal of this subsection is to define an extension of $H$ by an element $e^{-\frac12}$ together with a natural involution $e^{-\frac12}\to-e^{-\frac12}$ in a proper way (the result would be called $H[e^{-\frac12}]$).

 Let $H$ be an associative algebra and $e\in H$ be such that operator $$\adj e: H\to H\hspace{10pt}(x\to [e, x]:=ex-xe)$$ acts locally nilpotently on $H$. Then we set $\widehat{(H, ad e, \frac12)}$ to be the algebra which is as a vector space a direct sum of $(\widehat{(\mathbb Z, \frac12)}:=\mathbb Z\sqcup\{\mathbb Z+\frac12\})$-copies of $H$ with multiplication
$$(xt^n)\cdot(yt^m)=\sum\limits_{i=0}^{\infty}[{\small \left(\begin{array}{c}n\\i\end{array}\right)} x (\adj e)^i(y)t^{m+n-i}]$$
for $x, y\in H, m, n\in\widehat{(\mathbb Z, \frac12)}$, where we denote by $xt^n$ the element $x\in H$ in the $n$-th copy of $H$. We wish to note that even if formally the sum above is infinite it is essentially finite as $\adj e$ is a locally nilpotent operator. One can explicitly check that algebra $\widehat{(H, ad e, \frac12)}$ is associative.

One can think about $\widehat{(H, ad e, \frac12)}$ as about some power series extension of $H$/power series extension of a localization of $H$.

As a next step we figure out some two-sided ideal of $\widehat{(H, ad e, \frac12)}$, namely it is an ideal generated by $(e-t)$ but it is quite useful to write it down more explicitly. The following lemma is straightforward.
\begin{lemma}\label{Liv}a) Vector space $I$ spanned by
$$xt^{n+m}-xe^nt^m,\hspace{10pt}\mbox{for~all~}n\in\mathbb Z,~m\in\widehat{(\mathbb Z, \frac12)}$$
is a two-sided ideal of $\widehat{(H, ad e, \frac12)}$,

b) $I$ is generated as a two-sided ideal by $(e-t)$.
\end{lemma}
We denote by $H[e^{\pm\frac12}]$ quotient $\widehat{(H, ad e, \frac12)}/I$ and by $ev$ the respective map $ev: \widehat{(H, ad e, \frac12)}\to H[e^{-\frac12}]$. The reason for this notation is that $H[e^{-\frac12}]$ is generated by $$ev(H), ev(t^{\frac12}), ev(t^{-\frac12})$$ and $ev(t)=ev(e)$. The following lemma is implied by Lemma~\ref{Liv}.
\begin{lemma} a) For any $x\in H[e^{-\frac12}]$ there exists $x_1, x_2\in H, n\in\mathbb Z$ such that $$x=ev(x_1t^{-n}+x_2t^{-n-\frac12}).$$
b) For any $x_1, x_2, x_1', x_2'\in H, n\in\mathbb Z$ we have that
$$ev(x_1t^{-n}+x_2t^{-n-\frac12})=ev(x_1't^{-n}+x_2't^{-n-\frac12})$$
if and only if for some $m\in\mathbb Z_{\ge0}$ we have that $(x_1-x_1')e^m=(x_2-x_2')e^m=0$.\end{lemma}
\begin{corollary}If $H$ has no zero-divisors then $ev|_H: H\to H[e^{-\frac12}]$ is injective.\end{corollary}

It is clear that the algebra $H$ defined in Section~\ref{SW} has no zero-divisors and thus we can identify $H$ with it's image in $H[e^{-\frac12}]$. We also prefer to avoid whenever possible notation $ev(t^n)$ using $e^n$ instead. We left to mention that the linear map
$$\sigma_t(x_1t^m)\to (-1)^{2m}x_1t^m$$is an involution of algebra $\widehat{(H, ad e, \frac12)}$ which preserves $I$. Thus $\sigma_t$ defines an involution $\sigma_e: H[e^{-\frac12}]\to H[e^{-\frac12}]$ of $H[e^{-\frac12}]$ such that $\sigma_e(e^{-\frac12})=-e^{\frac12}$ and $\sigma_e$ preserves $H$ pointwise.
\subsection{``Almost'' decomposition of $H$}
Let $H$ be an associative algebra with no zero-divisors, $\{e, h, f\}\subset H$ be an $\frak{sl}_2$-triple, $\sigma: H\to H$ be an involution of $H$ satisfying conditions (1) and (2) of Lemma~\ref{Lsl2inv}. Then $\sigma$ gives rise to the automorphism $\widehat\sigma$ of $\widehat{(H, ad e, \frac12)}$ such that $$\widehat\sigma|_H=\sigma,\hspace{10pt} \widehat\sigma(t)=t.$$ This automorphism preserves $I$ and hence defines an automorphism (which we also denote $\sigma$) of $H[e^{-\frac12}]$.

\begin{lemma}\label{L1} The quotient field of $H$ is generated by $H^e, e$ and $h$.\end{lemma}
\begin{proof}The casimir $\theta=h^2+2(ef+fe)$ of $\Ue(\frak{sl}(2))=\Ue(\{e, h, f\})$ is contained in $H^e$ and hence $$f=\frac{\theta-h^2-2h}4e^{-1}$$ is contained in the envelop of $\{H^e, h, e\}$ in the quotient field of $A$. It follows from the $\frak{sl}(2)$-representation theory that $H^e$ and $f$ generate $H$ and thus generate the quotient field of $H$.\end{proof}
We have a natural adjoint action of $h$ on $H^e$. This action induces the natural nonnegative grading $\{(H^e)_i\}_{i\in\mathbb Z_{\ge0}}$ on $H^e$. We consider the map $\psi: H^e\to H[e^{-\frac12}]$ such that $a\to ae^{-\frac i2}$ for any $a\in(H^e)_i$. Clearly $\psi$ is a homomorphism of algebras and the image of $\psi$ commutes both with $e$ and $h$. We denote the image of $\psi$ by $(H^e)_\psi$.
\begin{lemma} Elements $(H^e)_\psi, e^{\frac12}, h$ generates the quotient field of $H[e^{-\frac12}]$. Moreover $$H[e^{-\frac12}]\approx\Ue(\{e^{\frac12}, h\})\otimes(H^e)_\psi.$$\end{lemma}
\begin{proof} First statement is implied by Lemma~\ref{L1}. Thus we now focus on the second statement. It is clear that we have a natural map
$$\gamma:~\Ue(\{e^{\frac12}, h\})\otimes(H^e)_\psi\to H[e^{-\frac12}]$$and it is enough to prove that $\gamma$ is injective. Hence it is enough to show that $$\sum\limits_{i, j\in \mathbb Z_{\ge0}}e^{i/2}h^j H_{ij},\hspace{10pt}H_{ij}\in (H^e)_\psi,$$equals zero if and only if $H_{ij}=0$ for all $i, j$. The action of $\adj h:=[h, \cdot]$ on both $\Ue(\{e^{\frac12}, h\})\otimes(H^e)_\psi$ and $H[e^{-\frac12}]$ is semisimple and it is clear that the $i$-th eigenspaces will maps to $i$-th eigenspace. Thus it is enough to show that
$$\sum\limits_{j\in \mathbb Z_{\ge0}}h^j H_{j},\hspace{10pt}H_j\in (H^e)_\psi$$
equals zero if and only if $H_{j}=0$ for all $i, j$. Assume on the contrary that there exist $H_0,..., H_s$ such that

$$H_0h^0+...+H_sh^s=h^0(H_0)+...+h^s(H_s)=0$$and not all $H_0,..., H_s$ equal 0. Without loss of generality we can assume that

1) $H_0\ne0$, $H_s\ne0$,

2) $s$ is the smallest possible among all such expressions.\\
Under these conditions we have that $s>0$. To proceed we need the following simple lemma a proof of which is left to the reader.
\begin{lemma}We have $e^{-1}he=h+2$. For any polynomial $p$ we have $e^{-1}p(h)e=p(h+2)$.\end{lemma}
According to this lemma we have
$$0=e^{-1}(H_sh^s+...+H_0h^0)e-(H_sh^s+...+H_0h^0)=2sH_s h^{s-1}+(...)h^{s-2}+....$$
It is clear that this new expression is of the same form but of smaller $h$-degree. This is a contradiction.
\end{proof}

\begin{corollary} We have that $H\approx (\Ue(\{e^{\frac12}, h\})\otimes (H^e)_\psi)^{\sigma_e}$ where $$\sigma_e(e^\frac12)=-e^\frac12, \sigma_e(h)=h,\hspace{10pt}\sigma_e(\psi(a))=(-1)^i\psi(a)\mbox{~for~}a\in (H^e)_i.$$\end{corollary}
\section{Proof of Theorem~\ref{T1}}\label{SprT}
If $\frak g$ is of type $A$ the statement of Theorem~\ref{T1} follows from results~\cite{FMO}. Thus we can focus on all other cases and apply results of Section~\ref{SinvH}. Lemma~\ref{Ls1} implies that it is enough to prove the following proposition.
\begin{proposition}\label{Pm}If $\frak g$ is a simple Lie algebra then
$$(\mathbb F[z_d, \partial_d]\otimes H)^{\tau'\otimes\sigma}\approx \mathbb F[z_d, \partial_d]\otimes H,$$
where $\tau'(z_d)=-z_d, \tau'(\partial_d)=-\partial_d$.
\end{proposition}

According to Section~\ref{SinvH} algebra $H[e^{-\frac12}]$ carries two involutions: $\sigma$ and $\sigma_e$. The involutions $\sigma_e$ and $\sigma$ commutes one with each other and thus $\sigma_e\sigma$ is also an involution of $H[e^{-\frac12}]$. We denote this involution $\sigma_U$. One can easily check that

1) $\sigma_U$ preserves $(H^e)_\psi$ pointwise and preserves $\Ue(\{e^\frac12, h\})$ as a set,

2) $\sigma$ preserves $\Ue(\{e^\frac12, h\})$ pointwise and preserves $(H^e)_\psi$ as a set,

3) $\sigma_e=\sigma_U\sigma=\sigma\sigma_U$.\\

Set $$W_x:=\mathbb F[x, \partial_x]:=\mathbb F\langle x, \partial_x\rangle/(x\partial_x-\partial_xx=-1)$$
(polynomial differential operators in one variable). We denote by $\sigma_x$ the involution of $W_x$ defined by the following formulas
$$\sigma_x(x)=-x, \sigma_x(\partial_x)=-\partial_x.$$
Similarly we define $W_y$ and $\sigma_y$. The quotient fields of algebras $W_y$ and $\Ue(\{e^\frac12, h\})$ are isomorphic via identification
$$y\leftrightarrow e^\frac12,\hspace{10pt} y\partial_y\leftrightarrow h,$$ and involutions $\sigma_y$ and $\sigma_U$ corresponds one to each other under this isomorphism.

We have that $H\approx H[e^{-1}]\approx(\Ue(\{e^\frac12, h\})\otimes (H^e)_\psi)^{\sigma_e}$ and thus $H\approx(W_y\otimes (H^e)_\psi)^{\sigma_y\times\sigma}$. Therefore Proposition\ref{Pm} is implied by the following lemma.
\begin{lemma}\label{Ldec}We have $(W_x\otimes W_y\otimes (H^e)_\psi)^{\{\sigma_x\times \sigma, \sigma_y\times \sigma\}}\approx(W_x\otimes W_y\otimes (H^e)_\psi)^{\sigma_y\times \sigma},$ where $\{\sigma_x\times \sigma, \sigma_y\times \sigma\}$ is a group of order 4 generated by $\sigma_x\times\sigma$ and $\sigma_y\times\sigma$.\end{lemma}
To prove Lemma~\ref{Ldec} we need another set of generators of the quotient field of $W_x\otimes W_y$. Namely we set

$$z:=\frac yx,\hspace{10pt}\partial_z:=x\partial_y,\hspace{10pt} \partial_x':=\partial_x+\frac yx\partial_y.$$
One can easily check that $z, \partial_z, x, \partial_x'$ generates the quotient field of $W_x\otimes W_y$ and that

$$[z, x]=[\partial_x', \partial_z]=[z, \partial_x']=[x, \partial_z]=0,\hspace{10pt}[x, \partial_x']=[z, \partial_z]=-1.$$
 Put $W_x':=\mathbb F[x, \partial_x'], W_z:=\mathbb F[z, \partial_z]$. We have that $\mathbb F[x, \partial_x', z, \partial_z]\cong W_x'\otimes W_z$ and that the quotient field of $W_x\otimes W_y$ is identified with the quotient field of $W_x'\otimes W_z$ under this isomorphism. We have that
$$\sigma_y(z)=-z, \sigma_y(\partial_z)=-\partial_z, \sigma_y(x)=x, \sigma_y(\partial_x')=\partial_x'$$and$$\sigma_x\sigma_y(z)=z, \sigma_x\sigma_y(\partial_z)=\partial_z, \sigma_x\sigma_y(x)=-x, \sigma_x\sigma_y(\partial_x')=-\partial_x'.$$
We set $\sigma_x':=\sigma_x\sigma_y$ and $\sigma_z:=\sigma_y$ (this notation is justified the formulas above).
\begin{proof}[Proof of Lemma 1.5.]We have that \begin{center}$(W_x\otimes W_y\otimes (H^e)_\psi)^{\{\sigma_x\times \sigma, \sigma_y\times \sigma\}}=(W_x\otimes W_y\otimes (H^e)_\psi)^{\{\sigma_x\times\sigma_y, \sigma_y\times \sigma\}}=$\\$=(W_x\otimes W_y\otimes (H^e)_\psi)^{\{\sigma_x', \sigma_y\times \sigma\}}\approx (W_x'\otimes W_z\otimes (H^e)_\psi)^{\{\sigma_x', \sigma_z\times\sigma\}}=(W_x')^{\sigma_x'}\otimes (W_z\otimes (H^e)_\psi)^{\sigma_z\times\sigma}$.\end{center}
We left to mention $(W_x')^{\sigma_x'}\approx W_x$ and $(W_z\otimes(H^e)_\psi)^{\sigma_z\times\sigma}\cong(W_y\otimes(H^e)_\psi)^{\sigma_y\times\sigma}\approx A$.
\end{proof}
\section*{Acknowledgements}
I would like to thank A.~Premet for many useful discussions on $W$-algebras which help me to write this paper.

\end{document}